\documentclass{article}

\usepackage[utf8]{inputenc}
\usepackage{balance}
\usepackage{fullpage}
\usepackage{authblk}
\usepackage{todonotes}
\usepackage{hyperref}
\usepackage{amsfonts,mathtools,amssymb,amsthm} 
\usepackage{cleveref}
\usepackage[caption=false, font=normalsize, labelfont=sf, textfont=sf]{subfig}

\DeclareMathOperator*{\argmin}{argmin}

\newtheorem{definition}{Definition}
\newtheorem{theorem}{Theorem}

\newtheorem{lemma}{Lemma}

\newtheorem{proposition}{Proposition}

\newcommand{\indic}{\mathbbm{1}}

\newcommand{\R}{\mathbb{R}}
\newcommand{\PowerSet}[1]{2^{#1}}
\newcommand{\N}{\mathbb{N}}

\newcommand{\Links}{\mathcal{A}}
\newcommand{\Vertices}{\mathcal{V}}
\newcommand{\Graph}{G}
\newcommand{\linkOne}{\ell}
\newcommand{\linkTwo}{\linkOne'}

\newcommand{\destination}{d}

\newcommand{\SetOfTime}{\mathcal{T}}
\newcommand{\EndTime}{T}
\newcommand{\waitingTime}{w}

\newcommand{\runCost}{r}
\newcommand{\travelTime}{c}

\newcommand{\Vehicles}{\mathcal{N}}
\newcommand{\NumOfPlayer}{N}
\newcommand{\vehicle}{i}
\newcommand{\strategy}{\pi}
\newcommand{\Strategies}{\Pi}
\newcommand{\policy}{\boldsymbol{\strategy}}
\newcommand{\Policies}{\boldsymbol{\Strategies}}
\newcommand{\policyProfile}{\underline{\policy}}
\newcommand{\policyProfiles}{\underline{\Policies}}
\newcommand{\RestrictedPolicies}{\widetilde{\Policies}}
\newcommand{\RestrictedPolicyProfiles}{\widetilde{\policyProfiles}}
\newcommand{\cost}{J}
\newcommand{\deviationIncentive}{D}
\newcommand{\averageDeviationIncentive}{\overline{\deviationIncentive}}

\newcommand{\StateDistribution}{m}

\newcommand{\playerState}{x}

\newcommand{\playerLocation}{\ell}

\newcommand{\gameLocation}{\underline{\playerLocation}}
\newcommand{\playerStates}{\mathcal{X}}
\newcommand{\playerTraj}{\boldsymbol{\playerStates}}
\newcommand{\gameStates}{\underline{\playerStates}}
\newcommand{\gameTraj}{\underline{\boldsymbol{\playerStates}}}

\newcommand{\ProbabilityMeasureSet}[1]{\mathcal{P}\left({#1}\right)}
\newcommand{\expectation}[2]{\mathbb{E}_{#1}\left[#2\right]}
\newcommand{\randomPlayerState}{X}
\newcommand{\randomPlayerLocation}{L}
\newcommand{\randomPlayerWaitingTime}{W}
\newcommand{\randomPlayerRouting}{U}
\newcommand{\randomPlayerDestination}{D}
\newcommand{\randomPlayerStateTraj}{\boldsymbol{\randomPlayerState}}

\newcommand{\randomPlayerWaitingTimeTraj}{\boldsymbol{\randomPlayerWaitingTime}}
\newcommand{\randomPlayerRoutingTraj}{\boldsymbol{\randomPlayerRouting}}

\newcommand{\randomGameState}{\underline{\randomPlayerState}}
\newcommand{\randomGameLocation}{\underline{\randomPlayerLocation}}
\newcommand{\randomGameWaitingTime}{\underline{\randomPlayerWaitingTime}}

\newcommand{\randomGameDestination}{\underline{\randomPlayerDestination}}
\newcommand{\randomGameStateTraj}{\underline{\randomPlayerStateTraj}}

\newcommand{\randomGameRoutingTraj}{\underline{\randomPlayerRoutingTraj}}

\newcommand{\linksDistribution}{\nu}
\newcommand{\linksDistributionTraj}{\boldsymbol{\linksDistribution}}

\newcommand{\pathsIndices}{\mathcal{K}}
\newcommand{\SetOfPaths}{Q}
\renewcommand{\path}{q}

\usepackage{bbm}

\title{Solving N-player dynamic routing games with congestion: a mean field approach}

\author[1, 2, 3]{Theophile Cabannes}
\author[5]{Mathieu Lauri{\`e}re}
\author[6]{Julien Perolat}
\author[5]{Raphael Marinier}
\author[5]{Sertan Girgin}
\author[7]{Sarah Perrin}
\author[5]{Olivier Pietquin}
\author[1]{Alexandre M. Bayen}
\author[3]{Eric Goubault}
\author[6]{Romuald Elie}
\affil[1]{University of California, Berkeley}
\affil[2]{Google Research}
\affil[3]{Laboratoire d’Informatique de l’Ecole Polytechnique, Ecole Polytechnique} 
\affil[5]{Google Research, Brain team}
\affil[6]{DeepMind}
\affil[7]{Univ. Lille, CNRS, Inria, Centrale Lille, UMR 9189 CRIStAL}

\date{October 2021}

\begin{document}

\maketitle

\begin{abstract}
The recent emergence of navigational tools has changed traffic patterns and has now enabled new types of congestion-aware routing control like dynamic road pricing.
Using the fundamental diagram of traffic flows -- applied in macroscopic and mesoscopic traffic modeling -- the article introduces a new $\NumOfPlayer$-player dynamic routing game with explicit congestion dynamics.
The model is well-posed and can reproduce heterogeneous departure times and congestion spill back phenomena.
However, as Nash equilibrium computations are PPAD-complete, solving the game becomes intractable for large but realistic numbers of vehicles $\NumOfPlayer$.
Therefore, the corresponding mean field game is also introduced. 
Experiments were performed on several classical benchmark networks of the traffic community: the Pigou, Braess, and Sioux Falls networks with heterogeneous origin, destination and departure time tuples.
The Pigou and the Braess examples reveal that the mean field approximation is generally very accurate and computationally efficient as soon as  
the number of vehicles exceeds a few dozen.  
On the Sioux Falls network (76 links, 100 time steps), this approach enables learning traffic dynamics with more than 14,000 vehicles.
\end{abstract}

\section{Introduction}

\subsection{Motivations}
In 2019, the Texas A\&M Transportation Institute estimated that the U.S. loses \$166 billion per year due to the impact of congestion on fuel usage and productivity loss \cite{TTI2019urban}. 
The average auto commuter spends 54 hours in congestion and wastes 21 gallons of fuel every year due to congestion at a cost of \$1,010 in wasted time and fuel.
With the emergence of navigational applications, the traffic patterns have evolved due to congestion-aware routing behaviors \cite{cabannes2018impact}.
Being able to model traffic and especially routing choice adequately would enable traffic control to leverage the new routing behaviors in order to improve the network efficiency.

However, solving realistic routing choice problems requires the consideration of large scale multi-agent systems, where the number of vehicles making strategic decisions is not tractable within classical algorithms \cite{daskalakis2009complexity}.
This work focuses on finding a scalable approach to model the routing behavior of each vehicle in a dynamic road traffic environment as a large multi-agent dynamic system.

\subsection{Background}

\subsubsection{Dynamic traffic assignment}

Dynamic traffic simulations are used to model the evolution of the locations of the vehicles in the road network across time.
Within dynamic models, simulations are divided between macro, meso and microsimulations.
Microsimulations model each individual vehicle and its interaction with others.
At the other extreme, macrosimulations model traffic dynamics based on the flow of cars. 
Mesosimulations also model individual vehicles, but only consider the interactions between each vehicle and the flow (or volume) of vehicles.
In general, this simulations have no spatial resolution on links  
and are mostly discrete-event models.
In this simulations, a vehicle is slowed down if the volume of vehicles on the road section is large.
This is modeled using fundamental diagrams of the traffic flow \cite{siebel2006fundamental} -- a congestion function mapping the volume on each link to the link travel time.

In routing models, vehicles are routed based on the notion of Wardrop equilibrium \cite{Wardrop1952},
in which the travel times of every used path between the same origin and the same destination for a given departure time are equal and smaller than the one on any unused path. 
To find the route assignment that leads to a Wardrop equilibrium, routing models are divided between one-shot assignments (that assign the routes once for all using a stochastic model and the current traffic information) and iterative assignments (that assign the routes, run a simulation, access gap to the user equilibrium and update the routes accordingly until convergence).
To assess how far a traffic state is from the user equilibrium, the relative gap to the user equilibrium is considered \cite{chiu2011dynamic}. 
In fact, the Wardrop equilibrium condition interprets as a Nash equilibrium condition between the vehicles \cite[chapter 18]{Nisam2007} while the relative gap to the user equilibrium is similar to the average deviation incentive or the average marginal regret \cite{Cabannes2019regret} over the vehicles.

The most popular routing game model is arguably the non-atomic static routing game as presented e.g. in~\cite[chapter 18]{Nisam2007}.
This non-atomic game is a potential game \cite{monderer1996potential} that can be solved through convex programming \cite{Rosenthal1973, Patriksson2015}, enabling fast computation of the route choices for large networks.
In the static setting, vehicles try to minimize their travel cost by choosing a path from their origin to their destination. 
Path choices are converted into volume on each path, which translates into a volume on each link. Link volume are converted to link cost using congestion functions \cite[Table 1.1]{Patriksson2015}.

In static routing games, the congestion on each link does not evolve.  
These games cannot replicate dynamic phenomena like departure time choice and congestion spill back, and dynamic extensions of the static routing game have been introduced.
First, using the potential formulation of the static traffic assignment, dynamic traffic assignment has been defined as the solution of a dynamical variational inequality \cite{friesz1993}.
A similar approach allows to consider both routing choices and departure times \cite{han2013existence}.   
However, in such dynamical models, the game theoretical aspect is not explicit.
In parallel, dynamic routing games have been defined \cite{wie1993differential}, and their resolution via multi-agent reinforcement learning has been studied \cite{Shou2020}.
Nevertheless, current multi-agent learning algorithms do not scale in terms of population size 
\cite{daskalakis2009complexity}.
This inherently calls for new methods to study routing choice problems with very a large number of vehicles.
We next explain how a Mean field game perspective can help making an important step in this direction.

\subsubsection{Mean field game}

Mean field games (MFGs) have been introduced in \cite{MR2295621,MR2346927-HuangCainesMalhame-2006-closedLoop} to model differential game dynamics between infinitely many players with symmetric interactions.  
Similar to mean-field theory from statistical physics, the key idea is to use a macroscopic approximation of a large population with anonymous and symmetric players.
When the population is infinite, each player has no influence on the population distribution.
So in an MFG, one does not need to study the pairwise interactions between all players but simply the interaction between a representative infinitesimal player and the full population distribution of states (and possibly actions). Because players are assumed to be identical, it suffices to determine the strategy of a representative player in response to the full population behavior.
From a mathematical point of view, the solution of an MFG can be characterized by a coupled system of a forward equation for the population evolution and a backward dynamic programming one for the player's value function. This system is easier to solve than the Nash equilibrium of a finite-player dynamic game with a large number of players, which leads to a large system of coupled Bellman equations.
Intuitively, this is true when representing each player is more costly than representing a distribution over all possible states, which is the case for instance when the number of players is larger than the number of possible states.

Applications of MFGs include crowd modeling \cite{LachapelleWolfram-2011-MFG-congestion-aversion,MR3763083,achdou2019mean}, energy management \cite{samarakoon2015energy,li2016mean,alasseur2020extended}, epidemiology \cite{doncel2020meansir,elie2020contact,aurell2020optimalincentives} or financial markets \cite{lachapelle2016efficiency,cardaliaguet2018mean,elie2020large}.  In traffic theory, similar approaches model cars evolution over the network as traffic flow.
Modeling microscopic cars on a link as a macroscopic traffic flow has been interpreted as a mean field game e.g. in \cite{chevalier2015micro}.  
Connections between an MFG model with myopic players and the Lighthill-Whitham-Richards (LWR) model \cite{lighthill1955kinematic} on a single road have been studied~\cite{huang2019gameautomousvehicles}.
A recent literature \cite{salhab2018meanroute,tanaka2020linearly} builds on the construction of MFGs on graphs \cite{gueant2015existence} and develops the application of MFGs to road traffic management.

\subsection{Related work}

Several works study mean field routing games. First, continuous time models have been studied. 
The existence and uniqueness of the Nash equilibrium of a MFG with congestion on a graph has been shown in \cite{gueant2015existence}, where the state space is the set of nodes. Models in which the state space is given by the edges have been analyzed e.g. in~\cite{achdou2020finitehorizonnetworkMFG}, which proved existence and uniqueness for a forward -backward system of equations with suitable conditions at the vertices of the network.   
In~\cite{bauso2016densitynetwork}, the authors analyzed an MFG model for traffic flow on networks by using an extended state space that includes the distribution of players on the network and they studied Wardrop equilibria.
In existing discrete time MFG models for routing, the players move one edge per time step and pay a cost that increases with the proportion of players on the same edge.
In~\cite{salhab2018meanroute}, the authors analyzed an MFG model and studied the impact of adding or removing edges on the equilibrium traffic flow. Their work provides a discrete time resolution of a mean field routing game with on an 11 link network with 6 time steps. 
In~\cite{tanaka2020linearly}, the authors proposed an MFG model that reduces to a linearly solvable Markov Decision Process and showed connections with Fictitious Play \cite{brown1951iterative} in some cases.  

The fact that existing models take congestion into account only through the cost functions  leads to paradox such as an ambiguity about the definition of travel time: the graph traversal time and the player cost can differ. Such issues make these models hardly applicable for traffic engineering.
Also, the main motivation for using an MFG-based routing method is to obtain an efficient equilibrium policy in the finite-player routing game, which has not been checked in existing works.

\subsection{Contributions and outline} 
The first main contribution of this work (see \cref{sec:dynamics}) is to  
introduce a novel dynamic mesoscopic traffic model viewed as a dynamic routing game with explicit congestion dynamics, i.e., congestion effects directly in the state evolution.
We first propose a finite-player game which can easily be interpreted in terms of vehicles, and then derive the corresponding MFG.
It is proved that Nash equilibria exist for both games. Furthermore, theoretical arguments supporting the MFG approximation are discussed. 

The second main contribution (see \cref{sec:experiment}) is to demonstrate numerically that the MFG provides an efficient way to approximately solve the finite-player routing game with very large population: the MFG is much less costly to solve and yet provides a very good approximate Nash equilibrium policy.  
This is illustrated on small networks for which baselines are available and on the Sioux Falls network -- a real network with 76 links and 100 time steps with 14,000 vehicles. Although this network is often used as a benchmark in the literature, to the best of our knowledge, it is the first time that a method is able to solve a dynamic routing game with a very large number of vehicles on this network.

\section{Dynamic N-player and mean field routing games}\label{sec:dynamics}

This section introduces the dynamic routing game and the corresponding MFG.
We show that the MFG approach allows to recover the dynamic routing game Nash equilibrium with a very large number of vehicles. Before moving to the mathematical details, the model we propose can be summarized in the following way. This dynamic routing game models the evolution of $\NumOfPlayer$ vehicles on a road network.
The vehicles are described by their current link location, the time they will spend on the link before exiting it, and their destination.
The action of a vehicle is the successor link they want to reach when exiting a given link.
Actions are encoded as integers from $0$ to $K$.
Pure actions for a player on link $\linkOne$, with a negative waiting time are the successors link of $\linkOne$.
When arriving on a link, the waiting time of the player is assigned based on the number of players on the link at this time.
As time goes by, the waiting time of a vehicle decreases until it becomes negative, then the vehicle moves to a successor link and the waiting time gets reassigned.
The total cost for the vehicle is its travel time. In the corresponding MFG, the vehicles of the $\NumOfPlayer$-player game are replaced by a representative vehicle and the probability distribution of the vehicles states.

\vskip 6pt

\noindent{\bf Notation.}  
We denote by $A^B$ the set of functions from a set $B$ to a set $A$. If $B$ is countable, $A^B$ interprets as a set of sequences indexed by elements of $B$. We use bold letters for functions of time and underlines for vectors whose coordinates correspond to players. The set of probability measures on a space $X$ is denoted $\ProbabilityMeasureSet{X}$. Unless otherwise specified, random variables are denoted by capital letters, functions of time are denoted by bold letters, and vectors indexed by the set of players are denoted by an underlined letter.

\subsection{Network and game set up}

Time is represented as an interval $\SetOfTime=[0, \EndTime]$ of $\R$.
The road network is described by a directed graph $\Graph = (\Vertices, \Links)$, where $\Vertices$ and $\Links$ respectively denote the sets of vertices and links of the road network.
When exiting a link $\linkOne\in\Links$, a vehicle chooses one of the possible successor links.
In case the link has no successor, the vehicle stays on the link until the end of time.
When joining a link, a vehicle get assigned a travel time on this link, that depends on the volume of traffic on the link.
More specifically, congestion induces a travel time spent on link $\linkOne\in\Links$ which is a function $\travelTime_{\linkOne} \in \R_{>0}^{[0,1]}$ of the proportion of vehicles on link $\linkOne$. We assume that $\travelTime_{\linkOne}$ is continuous.  
The congestion functions $(\travelTime_{\linkOne})_{\linkOne\in\Links}$ encode the heterogeneity of the roads' sensitivity to traffic volume within the network.
The following typical congestion function is based on an example provided by the 1964 traffic assignment manual of the U.S. Bureau of public road functions, see \cite[table 1.1]{Patriksson2015}: $\travelTime_{\linkOne}:\mu\mapsto t_0(1+\alpha (\mu/\mu_{\linkOne, c})^\beta)$, where $\alpha$ and $\beta$ are positive constants, $t_0$ is the free flow travel time (i.e., the travel time when the link is empty), and $\mu_{\linkOne, c}$ is the relative capacity of the link $\linkOne$ (which, in our context, is to be understood as a capacity in terms of proportion of players). 

Let us stress that the congestion functions are functions of the \emph{proportion} of vehicles within the link. However, in practice, congestion effects scale with the \emph{number} of vehicles (as well as other factors such as the width and the length of the road, which are considered fixed for a given network). One should thus interpret $\travelTime_\linkOne$ as being tuned for a given number of vehicles, say $N_0$. 
Concretely, in the previous example, if $n$ vehicles out of $N_0$ are on link $\ell$, we have $\travelTime_{\linkOne}(n/N_0) =  t_0(1+\alpha (n / (N_0\mu_{\linkOne, c}))^\beta)$. For $N_0' \neq N_0$, we have $\travelTime_{\linkOne}(n/N_0') =  t_0(1+\alpha (n/(N_0'\mu_{\linkOne, c}))^\beta)$. So the travel time can be viewed as a function of the number of vehicles $n$ on link $\linkOne$, provided the relative capacity is scaled by the total number of vehicles.

\subsection{N-player dynamic routing game}

Given the above network, this subsection defines a finite-player game.
Most of the notations are chosen to ease the presentation of the MFG in the next subsection.

\subsubsection{Traffic flow environment}

For the sake of convenience, we assume that the time horizon $T$ is large enough so that any driver will have time to travel through the network. 

Let $N$ be a positive integer.
The set of players is $\Vehicles = \{1, 2, \dots, \NumOfPlayer\}$.
The number of players in the game is not necessarily the total number of vehicles $N_0$ in the real-life scenario.
Each player of the model corresponds to a proportion of the real number of vehicles, which allows to define a player as an infinitesimal portion of flow that does not impact network travel time in the MFG.
In the case where $N=N_0$, a player is a vehicle.
Player $\vehicle \in \Vehicles$ starts at an origin link $\randomPlayerLocation^\vehicle_0 \in \Links$ with a departure time $\randomPlayerWaitingTime^\vehicle_0 \in \SetOfTime$, and has a destination link $\randomPlayerDestination^\vehicle_0 \in \Links$. 
This is the initial state of the player. 
Intuitively the player wants to start moving at time $\randomPlayerWaitingTime^\vehicle_0$ from $\randomPlayerLocation^\vehicle_0$ and tries to reach $\randomPlayerDestination^\vehicle_0$. We assume that the players' initial state are distributed according to a finite-support distribution $\StateDistribution_0$. 
Both the origin and the destination are modeled as links, so that the location of the vehicle is always described as a link.
In experimental setups, a origin link is added before each origin node and a destination link is added after each destination node.
Being on the origin link means having not departed yet, and being on the destination link means having finished their trip.

Then, at any time step $t$, the state of a player $\vehicle$ is not only the link $\randomPlayerLocation^\vehicle_t$ where they stand, but also their waiting time $\randomPlayerWaitingTime^\vehicle_t$ before exiting this link together with their destination $\randomPlayerDestination^\vehicle_t$.
$\randomPlayerLocation^\vehicle_t$ and $\randomPlayerWaitingTime^\vehicle_t$ are random variables due to the randomness in the action choices.  
Even though the destination is constant through time ($\randomPlayerDestination^\vehicle_t = \randomPlayerDestination^\vehicle_0$ for all $t$), including this information in the state allows to keep track of the objective in the player's policy.  
So the state space for each driver is $\playerStates = \Links\times\SetOfTime\times\Links$, where the first component is for the current location and the last one is for the destination (recall that the destination is represented by a link in our model). 
Then, the space of vehicle trajectories is $\playerTraj = \playerStates^{\SetOfTime}$.
The state trajectories are in the space of triples (location, waiting time, destination), which provide more information than the physical trajectories just in terms of locations. 
At the population level, the states of all the agents is a vector $\randomGameState = (\randomPlayerState^\vehicle)_{\vehicle\in\Vehicles}$. 
The state space for the whole population is $\gameStates = \playerStates^{\Vehicles}$, and the corresponding space of trajectories is $\gameTraj = \playerStates^{\SetOfTime\times\Vehicles}$. 
We respectively call game state and game trajectory the state  and trajectory of the population.  

\subsubsection{Routing policy}

When at link $\linkOne$, a player can try to move to another link among the successors of $\linkOne$, and the transition is realized provided the waiting time is $0$. 
The players are allowed to randomize their actions. 
We thus call strategy function and denote by $\strategy$ a function from $\playerStates$ to $\ProbabilityMeasureSet{\Links}$ such that for any $\playerState = (\linkOne, \waitingTime, \destination)$, $\strategy(\playerState)$ has support in the successors of $\linkOne$.
We denote by $\Strategies$ the set of such strategy functions.
A (feedback or closed-loop) policy $\policy$ is a function that associates to each time a strategy function, so it is an element of the set $\Policies = \Strategies^{\SetOfTime}$ of policies. 
The notation $\policy_t(\linkTwo | \linkOne, \waitingTime, \destination)$ represents the probability at time $t$ with which the agent would like to go from $\linkOne$ to $\linkTwo$  given the fact that their waiting time is $\waitingTime$ and their destination is $\destination$.
A policy profile $\policyProfile$ is a vector of policy functions with one policy for each player, i.e., it is an element of  $\policyProfiles = \Policies^{\Vehicles}$. 
Studying this class of policies can be justified by the fact that it allows each player to take a decision based only on their own state, which is realistic if the players do not know the situation of the rest of the population. 
More information could be added in the inputs of the policy (e.g., the proportion of agents on the current link), but this is beyond the scope of this work.

\subsubsection{State dynamics}
Since the players' initial states and actions are randomized, their trajectories are stochastic.
Given a policy profile $\policyProfile\in\policyProfiles$, $\randomGameState_t=(\randomGameLocation_t, \randomGameWaitingTime_t, \randomGameDestination_t) \in \gameStates$ denotes the random variable corresponding to the links, waiting times and destinations for all the players at time $t\in\SetOfTime$.
The stochastic process of the population state is denoted $\randomGameStateTraj = (\randomGameState_t)_{t \in \SetOfTime} \in \gameTraj$.
As indicated above, the players' interactions are through the travel time functions $(\travelTime_\linkOne)_{\linkOne \in \Links}$ taking into account congestion levels.  
So the interaction between a driver and the rest of the vehicles is only through the proportion of vehicles on the same link. 
It is thus convenient to introduce the empirical distribution $\linksDistribution^\NumOfPlayer_{\gameLocation} \in \ProbabilityMeasureSet{\Links}$ corresponding to a location profile $\gameLocation = (\linkOne^\vehicle)_{\vehicle \in \Vehicles} \in \Links^\NumOfPlayer$: for every $\linkOne \in \Links$, $\linksDistribution^\NumOfPlayer_{\gameLocation}(\playerLocation') = \frac{1}{N} \#\{\vehicle \,|\, \playerLocation^\vehicle = \playerLocation'\} \in [0,1]$,  
which is the proportion of players on the link $\playerLocation'$, given $\gameLocation$. This is all the information one needs from the game state to compute the interactions between players at link $\playerLocation'$. Note that $\linksDistribution^\NumOfPlayer_{\gameLocation}(\playerLocation')$ is invariant by permutation of the components of the vector $\gameLocation$.

Let us fix a policy profile $\policyProfile\in\policyProfiles$. 
We denote by $\randomGameRoutingTraj$ the $\Links^{\SetOfTime\times\playerStates\times\Vehicles}$-valued random variable assigned to the probability distribution given by the policy profile: for each $(t, \playerState, \vehicle) \in \SetOfTime\times\playerStates\times\Vehicles$,  $\randomPlayerRouting_t^\vehicle(\playerState)$ is an $\Links$-valued random variable with distribution $\strategy_{t}^\vehicle(\cdot|\playerState)$. 

The evolution of the state of the game  $\randomGameStateTraj_t=(\randomGameLocation_t, \randomGameWaitingTime_t, \randomGameDestination_t)$ is given by the following dynamics.
At initial time, $(\randomPlayerLocation^\vehicle_0, \randomPlayerWaitingTime^\vehicle_0, \randomPlayerDestination^\vehicle_0)$, $\vehicle \in \Vehicles$ are given, and then the dynamics is:
\begin{align*}
    t_{k+1} &= t_k + \min \{ \randomPlayerWaitingTime^{\vehicle}_{t_k},\ \vehicle\in\Vehicles\}
    \\
    \randomPlayerLocation^{\vehicle}_{t_{k+1}} &=
    \begin{cases}
        \randomPlayerRouting^{\vehicle}_{t_{k+1}}(\randomPlayerState^{\vehicle}_{t_{k}}) & \text{ if } \vehicle\in I_{t_{k+1}}
        \\ \randomPlayerLocation^{\vehicle}_{t_{k}} &\text{ otherwise;}
    \end{cases}
    \\
    \randomPlayerWaitingTime^{\vehicle}_{t_{k+1}} &=
    \begin{cases}
        \travelTime_{\randomPlayerLocation^{\vehicle}_{t_{k+1}}}\big(
        \linksDistribution^\NumOfPlayer_{\randomGameLocation_{t_{k+1}}}(\randomPlayerLocation_{t_{k+1}}^\vehicle)\big) & \text{ if } \vehicle\in I_{t_{k+1}}
        \\ \randomPlayerWaitingTime^{\vehicle}_{t_{k}}- (t_{k+1} - t_k)  &\text{ otherwise;}
    \end{cases}
    \\
    \randomPlayerLocation^{\vehicle}_{t} &= \randomPlayerLocation^{\vehicle}_{t_{k}} \qquad \forall k, \forall t\in[t_k,t_{k+1}[, \forall \vehicle\in\Vehicles
    \\
    \randomPlayerWaitingTime^{\vehicle}_{t} &=\randomPlayerWaitingTime^{\vehicle}_{t_{k}}- (t - t_{k}) \qquad \forall k, \forall t\in[t_k,t_{k+1}[, \forall \vehicle\in\Vehicles
    \\
    \randomPlayerDestination^\vehicle_t 
    &= \randomPlayerDestination^\vehicle_0, \qquad t \in \SetOfTime,
\end{align*}
where $I_{t_{k+1}} := \{\vehicle\in\Vehicles,\ \randomPlayerWaitingTime^{\vehicle}_{t_k}+t_k-t_{k+1}=0\}$ and using $(t_k)_{k\in\N}$ the sequence of times where one of the vehicles changes link with $t_0 = 0$ and $t+k=T$ if all the players have arrived their destination. 
The destination is constant through time and is not affected by the policy's randomness. 
Note that $\randomPlayerRoutingTraj^\vehicle = (\randomPlayerRouting^\vehicle_t)_{t \in \SetOfTime}$ is defined for all $t$ but used only when the player moves from one link to the next one, i.e., when the waiting time has vanished.
This enables reducing any pure (i.e. deterministic) policy as a path choice.

\subsubsection{Cost function}
\label{sec:Nplayer-cost-function}
Given a policy profile $\policyProfile\in\policyProfiles$, the cost for player $i$ is the average arrival time which can be defined as:  
$$
    \cost^\NumOfPlayer_\vehicle(\policy^{\vehicle}, \policy^{-\vehicle})
    =\expectation{\policyProfile}{\min\{t\in\SetOfTime,\ \randomPlayerLocation^\vehicle_t=\randomPlayerDestination^\vehicle\}}
    =\expectation{\policyProfile}{\int_{t \in \SetOfTime} \runCost(\randomPlayerState^\vehicle_t)dt}
$$
where $\policy^{-\vehicle} = (\policy^{1},\dots,\policy^{\vehicle-1}, \policy^{\vehicle+1}, \dots,\policy^{\NumOfPlayer})$, and the instantaneous cost is defined as: for every $\playerState = (\linkOne,\waitingTime,\destination)$, 
$
    \runCost(\playerState) = \indic_{\linkOne \neq \destination}.
$
Note that the running cost is independent of the (rest of the) population state, contrary to other models for routing or crowd motion in which the interactions are not in the dynamics but in the cost function.

Furthermore, the population is homogeneous (all players have the same dynamics evolution and same running cost), and player $\vehicle$ interacts with the other players only through $\linksDistributionTraj^\NumOfPlayer$ and for this reason, the cost function $\cost^\NumOfPlayer_\vehicle$ does not depend directly on the index $\vehicle$ but only on $\policy^\vehicle$: as a function, $\cost^\NumOfPlayer_\vehicle = \cost^\NumOfPlayer_{\vehicle'}$ for all $\vehicle'$. The policy profile $\policy^{-\vehicle}$ for the rest of the population is used only to compute $\linksDistributionTraj^\NumOfPlayer=(\linksDistribution^\NumOfPlayer_t)_{t \in \SetOfTime}$. Although $\policy^{-\vehicle}$ is necessary, it is not sufficient because $\linksDistributionTraj^\NumOfPlayer$ is also influenced by the policy $\policy^\vehicle$ chosen by the player under consideration. However, the influence of each player decays as $\NumOfPlayer$ increases, which will be the basis for the mean-field approach presented in \S~\ref{subsec:MFGapprox}.

\subsubsection{Nash equilibrium}
Considering that all the players are individually optimizing their own cost leads to the following notion of solution for the game. We refer to e.g.~\cite{myerson2013game} for more details. 
\begin{definition}[Nash equilibrium]
    A Nash equilibrium is a policy profile $\policyProfile^\star = (\policy^{\vehicle\star})_{\vehicle \in \Vehicles}\in\policyProfiles$ such that: 
    \[
        \forall \vehicle\in\Vehicles,\ \forall \policy\in\Policies,\ \cost^\NumOfPlayer_\vehicle(\policy^{\vehicle\star}, \policy^{-\vehicle\star}) \leq \cost^\NumOfPlayer_\vehicle(\policy, \policy^{-\vehicle\star}).
    \]
\end{definition}
The following result says that, in our model, such equilibria exist. 
\begin{theorem}[Existence of $N$-player Nash equilibrium, Kakutani-Fan-Glisckberg theorem \cite{glicksberg1952further}]\label{thm:existence-N-player-eq}
Assuming the continuity of the cost function with respect to the policy profiles, there exists a Nash equilibrium in the $\NumOfPlayer$-player routing game.
\end{theorem}
\begin{proof}
The proportion of players on each link is always a multiple of $1/\NumOfPlayer$. Since the number of links and the time horizon are finite, there is a finite set of times at which a vehicle can switch link. The set of policy profiles can thus be restricted to a finite set.
Therefore, the game can be restated as a game with finite state and action spaces. Assuming the continuity of the cost function with respect to the policy, it has a Nash equilibrium.
Further development of the proof are provided in the appendix.
\end{proof}

Besides the above definition, another way to express that a policy profile $\policyProfile$ is a Nash equilibrium is to say that the deviation incentive is $0$ for every player, where the deviation incentive for player $i$ is: 
$$
    \deviationIncentive^\NumOfPlayer_\vehicle(\policy^\vehicle,\policy^{-\vehicle}) = \cost^\NumOfPlayer_\vehicle(\policy^\vehicle,\policy^{-\vehicle})-\argmin_{\policy' \in \Policies}\cost^\NumOfPlayer_\vehicle(\policy', \policy^{-\vehicle}).
$$ 
This also serves as a basis to assess the convergence of algorithms towards a Nash equilibrium using the average deviation incentive:
\begin{equation}
    \label{eq:avg-dev-incentive}
    \averageDeviationIncentive^\NumOfPlayer(\policyProfile) = \frac{1}{\NumOfPlayer} \sum_{\vehicle=1}^{\NumOfPlayer} \deviationIncentive^\NumOfPlayer_\vehicle (\policy^\vehicle,\policy^{-\vehicle}).
\end{equation}

\subsection{Mean Field approximation}
\label{subsec:MFGapprox}

As mentioned in the introduction, solving the above $\NumOfPlayer$-player game is infeasible as $\NumOfPlayer$ is very large. We thus turn to an MFG version of the above routing game, which can be used to provide approximate Nash equilibria and whose quality improves as $\NumOfPlayer \to +\infty$.
This is based on considering the interactions between a typical player and a distribution representing the rest of the population.
This is possible thanks to the anonymity and the symmetry in the interactions, which allows us to focus on symmetric Nash equilibria. Intuitively, the law of large numbers allows to consider the state distribution instead of a large number of random variables induced by it.

\subsubsection{Traffic flow environment}

The state of a typical player at time $t$ is a random variable denoted by $\randomPlayerState_t = (\randomPlayerLocation_t, \randomPlayerWaitingTime_t, \randomPlayerDestination_t)$ which takes  values in $\playerStates = \Links\times\SetOfTime\times\Links$. At time $0$, the population's state distribution is $\StateDistribution_0$ and is known to the players. 

\subsubsection{Routing policy}
The space of policies is still  $\Policies$. For a policy $\policy$, we denote by $\policy_{t}(\linkTwo | \linkOne, \waitingTime, \destination)$ the probability with which a typical player using policy $\policy$ would like to go from $\linkOne$ to $\linkTwo$ given that their waiting time is $\waitingTime$ and their destination is $\destination$. The routing random variable is denoted by $\randomPlayerRouting$.

\subsubsection{State dynamics}
Assume that an infinitesimal agent uses policy $\policy$ while the rest of the population uses $\policy'$. Let $\linksDistributionTraj = (\linksDistribution_t)_{t \in \SetOfTime} \in \ProbabilityMeasureSet{\Links}^{\SetOfTime}$ be the flow of distributions on $\Links$ induced by the population that uses $\policy'$. The evolution of a typical player's state is given by the following dynamics. Let $t_0 = 0$ and let $(\randomPlayerLocation_{0}, \randomPlayerWaitingTime_{0}, \randomPlayerDestination_{0})$ be a given initial state. Then, the dynamics follow:
\begin{align*}
    t_{k+1} &= \randomPlayerWaitingTime_{t_k}+t_k
    \\
    \randomPlayerLocation_{t_{k+1}} &= \randomPlayerRouting_{t_{k}}(\randomPlayerState_{t_{k}}) 
         \\
    \randomPlayerWaitingTime_{t_{k+1}} &=
        \travelTime_{\randomPlayerLocation_{t_{k+1}}}\left(
        \linksDistribution_{t_{k+1}}(\randomPlayerLocation_{t_{k+1}})\right)
            \\ 
    \randomPlayerLocation_{t} &= \randomPlayerLocation_{t_{k}} \qquad \forall k, \forall t\in[t_k,t_{k+1}[
    \\
    \randomPlayerWaitingTime_{t} &=\randomPlayerWaitingTime_{t_{k}}- (t_{k} - t) \qquad \forall k, \forall t\in[t_k,t_{k+1}[
    \\
    \randomPlayerDestination_{t} &= \randomPlayerDestination_0, \qquad t \in \SetOfTime.
\end{align*}
Here $(t_k)_{k\in\N}$ denotes the sequence of times where the representative player changes link (we take $t_k = T$ when there are no more changes), 
and $\linksDistribution_{t}(\playerLocation)\in[0,1]$ is the proportion of the mean field population on link $\playerLocation$ at time $t$. 

\subsubsection{Cost function}
The cost of the typical player using policy $\policy$ when the population uses policy $\policy'$ is defined as:
$$
    \cost(\policy,\policy') 
    =
    \expectation{\policy,\policy'}{\int_{t \in \SetOfTime} \runCost(\randomPlayerState_t)dt}
$$
where the state of the representative player $\randomPlayerStateTraj = (\randomPlayerState_t)_{t \in \SetOfTime}$ has the above dynamics with policy $\policy$, and the instantaneous cost function $\runCost$ is the same function as in the finite player game (see \S~\ref{sec:Nplayer-cost-function}).
Analogously to the $\NumOfPlayer$-player game, the  policy $\policy'$ is used only to deduce $\linksDistributionTraj = (\linksDistribution_t)_{t \in \SetOfTime}$ that appears in the evolution of $\randomPlayerWaitingTimeTraj$. 
So the cost function $\cost$ could alternatively be written as a function of $(\policy,\linksDistributionTraj)$ instead of $(\policy,\policy')$. 
In contrast with the finite player regime, we highlight that here $\policy'$ completely determines $\linksDistributionTraj$ because the player under consideration is infinitesimal and hence their policy $\policy$ does not affect the flow $\linksDistributionTraj$ of distributions of locations of the population.

\subsubsection{Nash equilibrium}

The counterpart of the $\NumOfPlayer$-player Nash equilibrium in the mean-field regime can now be introduced. 
\begin{definition}[Mean field Nash equilibrium (definition 3.1. of \cite{saldi2018markov})]
A mean field Nash equilibrium (MFNE) is a policy $\policy^\star\in\Policies$ such that: $\cost(\policy^\star,\policy^\star) \leq \cost(\policy',\policy^\star)$ for all $\policy'$, or equivalently:
$$ 
    \policy^\star \in \argmin_{\policy\in\Policies} \cost(\policy,\policy^\star).
$$ 
\end{definition}

Another way to express that $\policy\in\Policies$ is a MFNE is to say that the average deviation incentive $\averageDeviationIncentive$ vanishes, where:
$$
    \averageDeviationIncentive(\policy) =
    \cost(\policy,\policy) - \argmin_{\policy'\in\Policies} \cost(\policy',\policy).
$$

\begin{theorem}[Existence of mean field Nash equilibrium, Kakutani-Fan-Glisckberg theorem \cite{glicksberg1952further}]
\label{thm:existence-MFNE}
Assuming the continuity of the cost function with respect to the policy profiles, and assuming that the support of the initial distribution of the waiting time is a finite set, there exists a mean field Nash equilibrium. 
\end{theorem}
\begin{proof}
The set of pure policies for the representative player can be restricted to the choice of a path given a departure time.
This set is finite as long as the support of the initial distribution of waiting time is.
Therefore the argmin map (also called Best response map) is a Kakutani-Fan-Glisckberg map providing the continuity of the cost with respect to the policy of the representative player and of the mean field.
The proof is further developed in appendix.
\end{proof}

The continuity of the cost function with respect to the policies plays a crucial role. A counter-example of the existence of a Nash equilibrium with a discontinuous cost is shown in the appendix. 

One of the advantages of considering a mean field setting, is that any MFNE is automatically a dynamic Wardrop equilibrium. 
\begin{theorem}[Dynamic Wardrop equilibrium \cite{Wardrop1952}]\label{thm:equatlization-travel-time}
For any mean field Nash equilibrium, all induced trajectories of players with the same initial state (origin, waiting time, destination), have the same travel time (i.e., the same total cost). 
\end{theorem}
\begin{proof}
In case a trajectory used by the representative player has a higher travel time than another one, then the player has an incentive to deviate, and the game is not a Nash equilibrium.
The proof is further developed in the appendix.
\end{proof}

Any mean field Nash equilibrium policy $\policy^\star$ can be used by the players in an $\NumOfPlayer$-player game. Intuitively, the larger $\NumOfPlayer$ is, the closer the population is to the mean field regime. In fact, it can be shown under suitable conditions that $\policyProfile^\star = (\policy^\star,\dots,\policy^\star) \in \Policies^{\NumOfPlayer}$ is an approximate Nash equilibrium whose quality improves with $\NumOfPlayer$ in the sense that: 
$$
    \averageDeviationIncentive^\NumOfPlayer(\policyProfile^\star) \to 0, \qquad \hbox{as $\NumOfPlayer \to +\infty$.}
$$
So if all the agents use the mean field Nash equilibrium policy, then any single player's incentive to deviate decreases when the population becomes larger.  
For example, \cite{saldi2018markov} prove in their setting that: if $\policy^\star$ is an MFNE, then for every $\epsilon>0$, there exists $\NumOfPlayer_0 \in \mathbb{N}$ such that for every $\NumOfPlayer \ge \NumOfPlayer_0$, the $\NumOfPlayer$-player policy profile $(\policy^\star,\dots,\policy^\star) \in \Policies^{\NumOfPlayer}$ satisfies: $\averageDeviationIncentive^\NumOfPlayer(\policyProfile) \le \epsilon$.

Next, to illustrate this property in our model, an explicit computation is carried out in the simple Pigou network and then is empirically verified on both Pigou and Braess networks. 

\subsubsection{Mean field equilibrium policy in the $\NumOfPlayer$ player Pigou game}\label{subsec:pigou}

For the sake of illustration, we present a toy example for which the solution can be computed analytically.

The graph has 2 nodes and 2 parallel links, say $\linkOne, \linkTwo$, relating these 2 nodes. The cost function is: $\travelTime_{\linkOne}(x) = 0.5 T$, $\travelTime_{\linkTwo}(x) = x T$ for all $x \in [0.1]$. 
The departure time (initial waiting time) is the same for all the agents. 
The mean field Nash equilibrium can be computed and yields an equilibrium distribution with proportions $\linksDistribution_t(\linkOne) = \linksDistribution_t(\linkTwo) = 0.5$.
On the other hand, the Nash equilibrium for the $\NumOfPlayer$ player game is such that $\linksDistribution_t(\linkOne)$ the current proportion of players on $\linkTwo$ is included in $[\frac{1}{2}-\frac{1}{\NumOfPlayer},\frac{1}{2}]$. 
As detailed in the appendix, we can check that the average deviation incentive of the mean-field equilibrium policy in the $\NumOfPlayer$-player game is \[
    \frac{T}{\NumOfPlayer 2^{\NumOfPlayer}}\sum_{m=1}^\NumOfPlayer{\NumOfPlayer-1\choose m} \max\left\{\frac{\NumOfPlayer}{2}-m-1, m+1-\frac{\NumOfPlayer}{2}\right\},
\] which goes to $0$ when $\NumOfPlayer\to\infty$.

\section{Experiments} \label{sec:experiment}

This section shows experimentally that (1) computing the mean field equilibrium  
is easier than computing the $\NumOfPlayer$-player Nash equilibrium using state of the art algorithms (sampled counterfactual regret minimization \cite{zinkevich2007regret}) and (2) it gives an excellent approximation of the $\NumOfPlayer$-player equilibrium when $\NumOfPlayer$ is large (above 30 in the case of the Pigou \cite{pigou1920} and the Braess \cite{braess2005paradox} network).
The experiments also show that (3) online mirror descent algorithm \cite{Perolat2021scaling} enables computing the mean field equilibrium on the Sioux Falls network, a classic use case in road traffic network games, with 14,000 vehicles (across two origin-destination pairs) and realistic congestion function (from the open source dataset \cite{networkGithub}).

\subsection{Context}

All the experiments are conducted within the OpenSpiel framework \cite{lanctot2019openspiel}, an open source library that contains a collection of environments and algorithms to apply reinforcement learning and other optimization algorithms in games.
The code is publicly available on GitHub\footnote{\url{https://github.com/deepmind/open_spiel}}.

\textbf{Goal of the experiments.}
The experiments aim to show that the mean field equilibrium policy is faster to compute than the $\NumOfPlayer$-player policy and approximates well an equilibrium policy in the corresponding $\NumOfPlayer$-player game, showing that the mean field approach solves the curse of dimensionality regarding the number of players in $\NumOfPlayer$-player games.
Intuitively, the MFG approach is relevant when the number of possible states for any player is lower than the number of players.
In that case, computing the population's distribution probabilities over the possible states is faster than simulating each player trajectory.
The approximation is correct when representing the probability distribution over the state space is similar to representing the sum of each individual player random variable state, which is the case with large number of players thanks to the central limit theorem \cite{bertsekas2002introduction}.
In the MFG, heterogeneity between the players is encoded in the state, to use the same policy for each player without a loss a generality.
As an example, in our model, the destination of the player is represented in the state.

\textbf{Metrics.}
The quality of the approximation of the Nash equilibrium policy completed by the candidate policy is measured using the average deviation incentive defined in ~\eqref{eq:avg-dev-incentive} (also known as the average marginal regret \cite{Cabannes2019regret}, or the relative gap to the dynamic user equilibrium \cite{chiu2011dynamic} in traffic engineering).

\textbf{Implementation.} The $\NumOfPlayer$-player game is encoded as a simultaneous, perfect information, general sum  
game.
The corresponding MFG is encoded as a mean field, perfect information, general sum  
game.
OpenSpiel provides many algorithms to find Nash equilibria of simultaneous games or MFGs.
These algorithms include model-free algorithms such as Neural Fictitious Self-Play \cite{heinrich2016deep} and model-based algorithms such as Counterfactual Regret Minimization (and some variants) \cite{zinkevich2007regret} which we use to solve the $\NumOfPlayer$-player game.
The experiments solve the MFG using the online mirror descent algorithm \cite{Perolat2021scaling}.
The experiments performed in OpenSpiel use a fixed time discretisation.

\textbf{Networks.} As classical network games consider demand between nodes, we add artificial origin and destination links before and after each node in the network (Pigou \cite{pigou1920}, Braess \cite{braess2005paradox} and Sioux Falls).
This enables defining vehicle location only using links, and defining state of not having begun a trip and having finished it.

The \emph{Pigou network} \cite{pigou1920} has two links $\linkOne$, $\linkTwo$ and two nodes (an origin and a destination one) which come from the conversion of the origin and the destination nodes.
A time discretisation of 0.01, with a time horizon of 2 is used.
The cost functions are $\travelTime_{\linkOne}(x) = 2$, $\travelTime_{\linkTwo}(x) = 1 + 2x$ and all the demand leaves the origin link at time $0$ and head towards the destination link.

The \emph{Braess network} game is the dynamic extension of the game described in \cite{braess2005paradox}.
The network has 5 links $AB$, $AC$, $BC$, $BD$ and $CD$, one origin node $A$ converted to an origin link $OA$ and a destination node $D$ converted to a destination link $DE$.
The cost functions are $\travelTime_{AB}(x) = 1+x$, $\travelTime_{AC}(x) = 2$, $\travelTime_{AB}(x) = 0.25$, $\travelTime_{BD}(x) = 2$, $\travelTime_{CD}(x) = 1+x$.
All the demand leaves the origin link at time $0$ and head towards the destination link.
We use a time step of 0.05 and a time horizon of 5.

In the \emph{augmented Braess network} game, a destination link $CF$ is added to the network and 50 more vehicles leave the origin to $DE$ at time 0, 0.5 and 1, while 50 others leaves the origin to $CF$ at times 0 and 1, totaling 250 vehicles with 2 different destinations and 3 different departure times.

The \emph{Sioux Falls network} game is used by the traffic community for proof of concepts on network with around 100 links.
The network (76 links without the origin and destination links), the link congestion functions, and an origin-destination traffic demand are  open source \cite{networkGithub}.
As the classical routing game \cite[Chapter 18]{Nisam2007} is a static game, the demand is only a list of tuple origin, destination and counts, and does not provide any departure time.
We use the network data (including the congestion functions) and generate a demand specific to the game.
We model 7,000 vehicles departing at time 0 from node 1 to node 19, and 7,000 vehicles departing at time 0 from node 19 to node 1.
We use a time step of 0.5 and a time horizon of 50.

\subsection{Mean field game solves the curse of dimensionality in the number of players}

In this section, the mean field equilibrium policy is computed for both the Braess and the Pigou network games.
In addition to being considerably faster to compute compared to the $\NumOfPlayer$-player Nash equilibrium, the mean field equilibrium provides an excellent approximation when $\NumOfPlayer$ is above 30.

The evolution of the Braess mean field Nash equilibrium policy is given on \cref{fig:braess_dynamics}. 
The travel time on the three possible paths are equals, which encodes the Nash equilibrium condition of the MFG provided that the travel time on each link is a multiple of the time step, accordingly to \cref{thm:equatlization-travel-time}.

\begin{figure}
    \centering
    \subfloat[]{\includegraphics[width=0.3\linewidth]{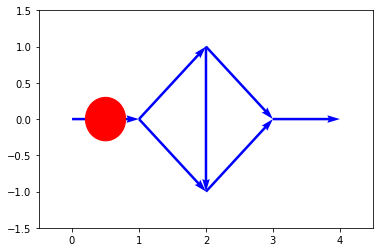}\label{subfig:t_0}}
    \subfloat[]{\includegraphics[width=0.3\linewidth]{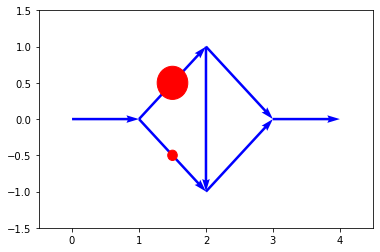}\label{subfig:t_025_175}}
    \subfloat[]{\includegraphics[width=0.3\linewidth]{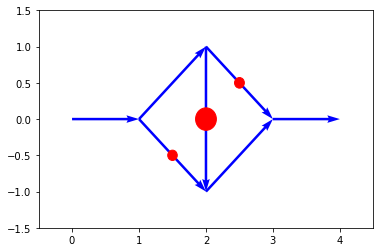}\label{subfig:t_200}}
    
    \subfloat[]{\includegraphics[width=0.3\linewidth]{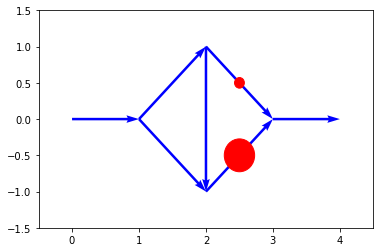}\label{subfig:t_225_375}}
    \subfloat[]{\includegraphics[width=0.3\linewidth]{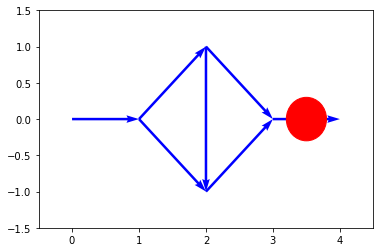}\label{subfig:t_400}}
    \caption{The dynamic of the Braess network in the mean field Nash equilibrium; the locations of the cars at time 0.0 (Figure \ref{subfig:t_0}), from time 0.25 to 1.75 (Figure \ref{subfig:t_025_175}), at time 2.0 (Figure \ref{subfig:t_200}), from time 2.25 to 3.75 (Figure \ref{subfig:t_225_375}), at time 4.0 (Figure \ref{subfig:t_400}). The travel time on each path are equal to 3.75, travel time equalization defines the mean field Nash equilibrium.}
    \label{fig:braess_dynamics}
\end{figure}

\subsubsection{While solving $\NumOfPlayer$-player game is intractable for large number of players, this can be done for the mean field game.}

We compare the running time of the algorithms for solving the $\NumOfPlayer$-player game and the mean field player game depending on the number of players it models.
The counterfactual regret minimization with external sampling (ext CFR) is used in the $\NumOfPlayer$-player game, as it is the fastest algorithms to solve the dynamic routing $\NumOfPlayer$-player game within the OpenSpiel library of algorithms (comparison done within the OpenSpiel framework are not reported here).
Online mirror descent (OMD) is used in the MFG.
Comparison between the running time of 10 iterations of ext CFR and OMD are done as a function of the number of vehicles modeled in \cref{fig:computation_time_algo_n_vs_mean_field}.
As the mean field Nash equilibrium does not depends on the number of vehicles the MFG models, the computation time of 10 iterations of OMD is independent of the number of vehicles modeled.
On the other hand, the computational cost of 10 iterations of ext CFR increases exponentially with the number of players, making the computation of a Nash equilibrium with a large number of players intractable with the algorithms of the OpenSpiel library.

\begin{figure}
    \centering
    \includegraphics[width=0.6\linewidth]{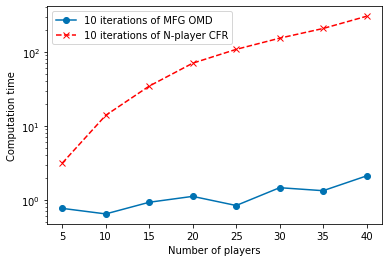}
    \caption{Computation time of 10 iterations of Online Mirror Descent in the MFG and of 10 iterations of sampled Counterfactual regret minimization as a function of the number of players $\NumOfPlayer$.}
    \label{fig:computation_time_algo_n_vs_mean_field}
\end{figure}

\subsubsection{The mean field equilibrium policy is a good approximation of the $\NumOfPlayer$-player equilibrium policy whenever $\NumOfPlayer$ is large enough.}

In the Pigou network game, the mean field equilibrium policy is almost a Nash equilibrium in the $\NumOfPlayer$-player game as soon as $\NumOfPlayer$ is larger than 20 players, see \cref{fig:pigou_mean_field_in_n_player}.
This was shown theoretically in \cref{subsec:pigou}, and is confirmed using approximate average deviation incentive of the mean field equilibrium policy in the $\NumOfPlayer$-player game.

\begin{figure}
    \centering
    \includegraphics[width=0.6\linewidth]{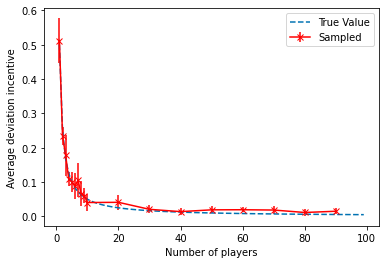}
    \caption{Average deviation incentive of the Nash equilibrium mean field policy in the $\NumOfPlayer$-player game as a function of $\NumOfPlayer$ in the case of the Pigou game.
    The sampled value is the value computed in OpenSpiel by testing all the possible pure best responses, and sampling game trajectories to get the expected returns.}
    \label{fig:pigou_mean_field_in_n_player}
\end{figure}

In the Braess network game, the mean field equilibrium policy is almost a Nash equilibrium in the $\NumOfPlayer$-player game as soon as $\NumOfPlayer$ is larger than 30 players.
The results are reported in the appendix.

\subsection{Mean field game approach can be extended to more complex set ups}

The MFG approach solves the curse of dimensionality in the number of players (whenever the number of possible states is much below the number of players).
It can be extended to more complex setups than the Pigou and Braess networks such as realistic traffic networks with demand.
The results obtained when using the MFG approach in the augmented Braess network game with heterogeneous departure time is reported in the appendix.
This section focuses on the extention of MFG approach to one of the classical benchmark network game used by the traffic community: the Sioux Falls network.

The experiment shows the ability to learn the mean field equilibrium policy on this 76 links network, with 14,000 vehicles going to two different destinations.
Using online mirror descent, we see that the average deviation incentive decreases to 1.55 (for a travel time of 27) over 100 iterations, see \cref{fig:omd_sioux_falls}.
We use a fixed learning rate of 1 in the 30 first iterations of the algorithm, 0.1 in the 31 to the 60 first iterations and a fixed learning rate of 0.01 in the 40 remaining iterations to produce \cref{fig:omd_sioux_falls}.

\begin{figure}
    \centering
    \includegraphics[width=0.6\linewidth]{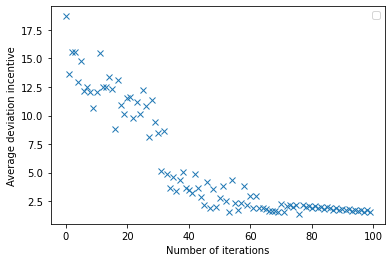}
    \caption{Online mirror descent average deviation incentive in the Sioux Falls MFG as a function of the number of iterations of the descent algorithm.}
    \label{fig:omd_sioux_falls}
\end{figure}

The resulting mean field policy is not exactly the Nash equilibrium policy of the MFG as its average deviation incentive is 1.55 (for a travel time of 27.5).
The game evolution displayed in \cref{fig:sioux_falls_dynamics} shows that some vehicles going from node 19 to node 1 have a longer travel time than others: on time step 26.5 (\cref{subfig:sf_53}) some vehicles have arrived to node 1 (top left) and some have not.

\begin{figure}
    \centering
    \subfloat[]{\includegraphics[width=0.3\linewidth]{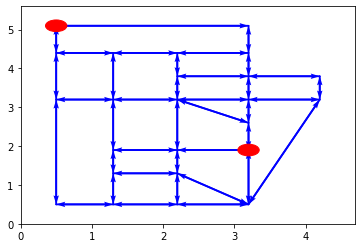}\label{subfig:sf_0}}
    \subfloat[]{\includegraphics[width=0.3\linewidth]{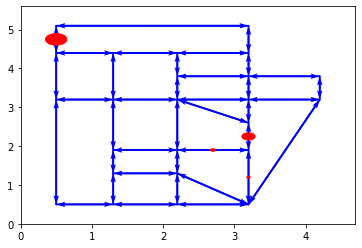}\label{subfig:sf_5}}
    \subfloat[]{\includegraphics[width=0.3\linewidth]{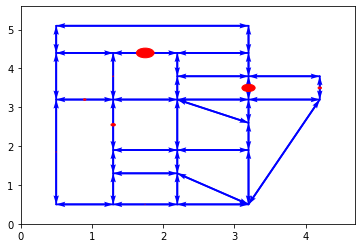}\label{subfig:sf_20}}
    
    \subfloat[]{\includegraphics[width=0.3\linewidth]{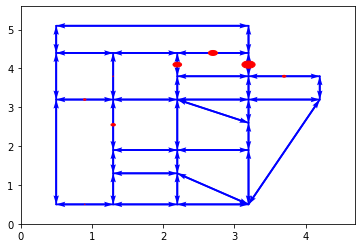}\label{subfig:sf_25}}
    \subfloat[]{\includegraphics[width=0.3\linewidth]{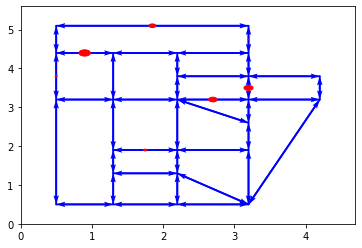}\label{subfig:sf_42}}
    \subfloat[]{\includegraphics[width=0.3\linewidth]{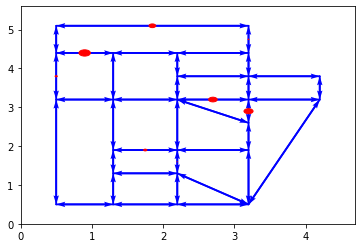}\label{subfig:sf_44}}
    
    \subfloat[]{\includegraphics[width=0.3\linewidth]{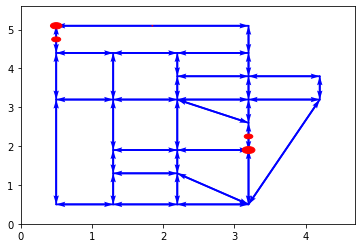}\label{subfig:sf_53}}
    \subfloat[]{\includegraphics[width=0.3\linewidth]{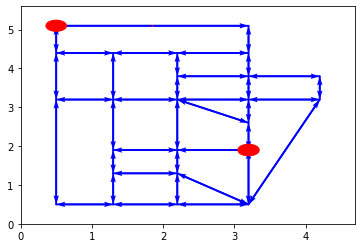}\label{subfig:sf_55}}
    \caption{Dynamics of the Sioux Falls network in the mean field Nash equilibrium. 
    Road network; location of the cars at time 0.0 (\ref{subfig:sf_0}), 2.5 (\ref{subfig:sf_5}), 10.0 (\ref{subfig:sf_20}), 12.5 (\ref{subfig:sf_25}), 21.0 (\ref{subfig:sf_42}), 22.0 (\ref{subfig:sf_44}), 26.5 (\ref{subfig:sf_53}), 27.5 (\ref{subfig:sf_55}). Some vehicles arrived at their destination after some that left the origin at the same time: the Nash equilibrium has not been reached. On average, players can expect saving 1.55 time by being the only one to be rerouted on a better path.}
    \label{fig:sioux_falls_dynamics}
\end{figure}

Average deviation of the learned mean field policy cannot be computed numerically in the 14,000 player game, due to the large number of players.

\section{Conclusion}

When using game theory models where agents make rational choices, to deal with the scabilabity in terms of number of agents, the main advantage of the mean field approach is that mean field Nash equilibria are easier to compute while providing a good surrogate for the equilibrium behavior in games with a finite but large number of players.
Detailed illustrations are provided on several numerical examples in this article.
In particular, besides toy-examples, this approach enables solving dynamic routing from a game perspective on a classical benchmark of the traffic community with 76 links and 14,000 vehicles.

This opens several directions for future work.
A natural next step would be to consider even more realistic scenarios, which means tackling more complex routing models and more complex networks. For example, it would be interesting to allow congestion spill back between links, and to propose a model ensuring the first-in first-out property according to which a vehicle can never exit a link earlier than another vehicle that arrived before them.
This could be done by including the game state, i.e., the whole state distribution, in the transitions and the policies.  
This will raise new challenges and new scalability issues in terms of models and networks.
To cope with this aspect, it would be interesting to combine the mean-field approach proposed in this work with state-of-the art reinforcement learning techniques.

\clearpage

\bibliographystyle{acm} 
\bibliography{bibfile.bib}

\clearpage 
\appendix 

\section{Complement on some theoretical aspects}

\subsection{Proofs}
\begin{proof}[Proof of Theorem~\ref{thm:existence-N-player-eq}]
The proof of the existence of a Nash equilibrium relies on the Kakutani-Glicksberg-Fan theorem \cite{glicksberg1952further} which needs the continuity of the cost function with respect to the policy profile. 

Because $\Vehicles$ and $\Links$ are finite and original waiting times are given, the time between two actions in the $N$-player game is necessary of the form 
$$
\sum_{\linkOne, k} \alpha_{\linkOne,k}\travelTime_{\linkOne}\left(\frac{k}{\NumOfPlayer}\right) + \sum_{\vehicle\in\Vehicles} \beta_\vehicle \randomPlayerWaitingTime_{0}^\vehicle\;,
$$
where $(\alpha_\linkOne)_{\linkOne\in\Links}\in\{-1,0,1\}^\Links,\ k\in\{0,\dots, \NumOfPlayer\}$ and  $\beta_\vehicle\in\{-1,0,1\}^{\NumOfPlayer}$. This generates at most \(3^{|\Links|(\NumOfPlayer+1) + \NumOfPlayer}\) possibilities. 
Therefore, there is a minimum time between two actions. 
As the time horizon $T$ is fixed, this implies that  there is a maximum number of times $M$ where an action can be taken.
Therefore, a policy only needs to define an action on all the possible tuples of times where an action should be taken.
The number $K$ of possible tuples is smaller than ${{3^{|\Links|(\NumOfPlayer+1) + \NumOfPlayer}}\choose{M}}$.
Without loss of generality, we restrict ourselves to the set of policies which is a subset of $\ProbabilityMeasureSet{\Links}^{\Links\times K}$, making the set of pure policies a subset of $\Links^{\Links\times K}$ finite, and the set of mixed-policies (its convex hull) a compact subset of the Euclidean vector space $\R^{\Links\times\Links\times K}$.
We will later show that any Nash equilibrium in the set of restricted mixed-policies is a Nash equilibrium in the set of non restricted mixed-policies.

We denote by $\pathsIndices\subseteq\N$ the set of pure policy indices.
We denote $\SetOfPaths=\{\path_k, k\in\pathsIndices\}$ the set of pure policies.
The set of mixed-policies is a subset of the simplex over the set of pure policy $\ProbabilityMeasureSet{\SetOfPaths}$.
Therefore, for all $ \policy\in\ProbabilityMeasureSet{\SetOfPaths}$, there exists $ \alpha\in\ProbabilityMeasureSet{\pathsIndices}$ identified with a vector of length $|\pathsIndices|$ such that $\policy=\alpha\cdot(\path_k)_{k\in\pathsIndices}$.
We denote by $\RestrictedPolicyProfiles\subseteq\ProbabilityMeasureSet{\SetOfPaths}$ the set of restricted mixed-policies and by $\RestrictedPolicyProfiles = \RestrictedPolicies^\NumOfPlayer$ the set of restricted mixed-policy profiles.

\begin{lemma}[Cost is convex combination of pure strategy cost]\label{cor:combination}
For any $\policyProfile\in S$, and for any $\vehicle\in\Vehicles$ with corresponding $\alpha\in\ProbabilityMeasureSet{\pathsIndices}$ such that $\policyProfile^{\vehicle}=\alpha\cdot(\path_k)_{k\in\pathsIndices}$
\[
    \cost^\NumOfPlayer_\vehicle(\policyProfile^{\vehicle}, \policyProfile^{-\vehicle}) = \sum_{k\in\pathsIndices} \alpha_k \cost^\NumOfPlayer_\vehicle(q_k, \policyProfile^{-\vehicle}).
\]
\end{lemma}

\begin{proof}
The equality is a direct consequence of the linearity of the expected value with respect to the probability distribution.
\end{proof}

Next, let the (set-valued) Best response map $\phi: \RestrictedPolicyProfiles \to \PowerSet{\RestrictedPolicyProfiles}$ be defined by for all: for $\policyProfile \in \RestrictedPolicyProfiles$,
\[
    \phi(\policyProfile) = \bigtimes_{\vehicle\in\Vehicles}\argmin_{\policy'\in \RestrictedPolicies}\cost^\NumOfPlayer_\vehicle (\policy', \policyProfile^{-\vehicle}).
\]
\begin{lemma}[Best response map has a closed graph]\label{cor:closed_graph}
$\phi$ has a closed graph.
\end{lemma}

\begin{proof}
For any sequence $(\policyProfile_j)_{j\in\N}\in \RestrictedPolicyProfiles^\N$ converging to $\policyProfile_\infty$ and any sequence $(x_j)_{j\in\N}\in \RestrictedPolicyProfiles^\N$ converging to $x_\infty$ such that $x_j\in \phi(\policyProfile_j)$ for all $j\in\N$, we have that, for any $\vehicle\in\Vehicles$:
\begin{align*}
    \cost^\NumOfPlayer_\vehicle(x_{j}^{\vehicle}, \policyProfile_{j}^{-\vehicle}) &\leq \alpha 
    \cdot \big(\cost^\NumOfPlayer_\vehicle(\path_k, \policyProfile_{j}^{-\vehicle})\big)_{k\in\pathsIndices} && \forall \alpha\in\ProbabilityMeasureSet{\pathsIndices},
\end{align*}
which holds in the limit provided that $\cost$ is continuous w.r.t. the policy profile (our regularity assumption).
Therefore $x_\infty\in\phi(\policyProfile_\infty)$ and $\phi$ has a closed graph.
\end{proof}

\begin{proposition}[Best response map has a fixed-point]
The map $\phi$  
is a Kakutani-Glicksberg-Fan map, therefore it admits a fixed-point \cite{glicksberg1952further}, which is a Nash equilibrium of the $N$-player game.
\end{proposition}

\begin{proof}\

The set $\RestrictedPolicyProfiles$ is:
\begin{itemize}
    \item non-empty (we assume the graph is non empty, so at least one path exists)
    \item compact (as a probability distribution over finite set)
    \item convex (as a probability distribution over finite set)
    \item subset of a Hausdorff locally convex topological vector space (as a probability distribution over finite set)
\end{itemize}
The function $\phi(\policyProfile)$ is non-empty and convex for all $\policyProfile\in\policyProfiles$ (because the minimization is a linear problem).
The function $\phi$ has a closed graph (\cref{cor:closed_graph}). Hence, by Kakutani-Glicksberg-Fan theorem \cite{glicksberg1952further}, $\phi$ has a fixed point.

\end{proof}

The proof of Theorem~\ref{thm:existence-N-player-eq} is concluded by noting that a fixed point $\policyProfile^\star\in \RestrictedPolicyProfiles$ of $\phi$ is a Nash equilibrium in $\policyProfiles$. Indeed, by definition of $\phi$:
\begin{align*}
    &\forall \vehicle\in\Vehicles,\ \forall \policy\in \RestrictedPolicies, \ 
    \cost^\NumOfPlayer_\vehicle(\policyProfile^{\vehicle\star}, \policyProfile^{-\vehicle\star}) \leq \cost^\NumOfPlayer_\vehicle(\policy, \policyProfile^{-\vehicle\star}).
\end{align*}
From here, by \cref{cor:combination} and by construction of $\RestrictedPolicies$:
\begin{align*}
    &\forall \vehicle\in\Vehicles,\ \forall \policy\in \Policies,\ \cost^\NumOfPlayer_\vehicle(\policyProfile^{\vehicle\star}, \policyProfile^{-\vehicle\star}) \leq \cost^\NumOfPlayer_\vehicle(\policy, \policyProfile^{-\vehicle\star}).
\end{align*}
The last line states that the Nash equilibrium in the set $\RestrictedPolicyProfiles$ is a Nash equilibrium in the set $\policyProfiles$ and finishes the proof of existence of the Nash equilibrium.

\end{proof}

\begin{proof}[Proof of Theorem~\ref{thm:existence-MFNE}]
This proof is similar to the proof of Theorem~\ref{thm:existence-N-player-eq} as long as the cost function is continuous with respect to the policy profile, which is assumed here.
In the mean field game, given a departure time the set of pure policies can be restricted to the set of path $\SetOfPaths^{\text{supp}(\randomPlayerWaitingTime_0)}$ given the departure time $\randomPlayerWaitingTime_0$, with $K$ the number of possible path (with less than a given number of cycles, which is possible due to minimum travel time on each link and finite time horizon) times the number of possible departure time.
Therefore a pure policy, should be understood as choosing a path $\path\in\SetOfPaths$ given a departure time $\randomPlayerWaitingTime_0$.
Encoding the policy in $\Policies$ that correspond to any choice of a path given a departure time might require some notation work, that we will not do here.
The reader should understand that the proof relies on finding a Nash equilibrium where the set of policies is the set of probability distributions on the set of paths given the departure time, and then showing that this Nash equilibrium policy can be translated into a policy in $\Policies$, and then showing that in the set of $\Policies$ the policy is still a Nash equilibrium.
To establish the existence of a Nash equilibrium in the set of probability distribution on the set of paths given a departure time $S=\ProbabilityMeasureSet{\SetOfPaths^{\text{supp}(\randomPlayerWaitingTime_0)}}$, 
the same arguments that the ones in the proof of Theorem~\ref{thm:existence-N-player-eq} can be used with the map $\phi: S \to \PowerSet{S}$ defined by:
\[\phi(\policy) = \bigtimes_{\vehicle\in\Vehicles}\argmin_{\policy'\in S}\cost(\policy', \policy)\]
which is a Kakutani-Fan-Glisckberg map.
Then, one can convert the path choice given a departure time to a list of actions in time, therefore convert it to a policy $\policy^\star$ in $\Policies$.
Then, one can show by contradiction that they cannot exist a policy $\policy$ in $\Policies$ that gives a strictly better outcome than the policy $\policy^\star$ in the cost function $\cost(\policy, \policy^\star)$.
Therefore $\policy^\star$ is also a Nash equilibrium in the set $\Policies$.
\end{proof}

\subsection{Counter-example of the existence without continuity}
Here we present a counter-example of the existence of a Nash equilibrium of the mean field game when congestion functions are not continuous.

Consider a network with two links, say $\linkOne$, $\linkTwo$, connecting the origin and destination links.  Let the congestion function be fined as:
\begin{align*}
    \travelTime_{\linkOne}(\mu) &= 
    \begin{cases}
        1 & \hbox{ if } \mu < 0.5 \\
        2 & \hbox{ otherwise}
    \end{cases}\\
    \travelTime_{\linkTwo}(\mu) &= 
    \begin{cases}
        1 & \hbox{ if } \mu \leq 0.5 \\
        2 & \hbox{ otherwise}
    \end{cases}
\end{align*}
Then the mean field game admits no Nash equilibrium.
Indeed, if $\linksDistribution_t(\linkOne)<0.5$, then $\travelTime_{\linkOne}=1$ and $\travelTime_{\linkTwo}=2$, and if $\linksDistribution_t(\linkOne)\geq 0.5$, then $\travelTime_{\linkOne}=2$ and $\travelTime_{\linkTwo}=1$.
In this case there is no flow allocation such that travel times are equal on the path that are used, therefore there is no Nash equilibrium.
This is due to the discontinuity of the cost of pure actions with respect to the state distribution.

\begin{proof}[Proof of Theorem~\ref{thm:equatlization-travel-time}]
\Cref{cor:combination} is still valid in the mean field game setup, with the set of pure policies corresponding to a choice of paths.
This implies that if the equilibrium policy $\policyProfile^\star$ is a mix-policy; i.e. there exist $ \pathsIndices^\star \subset \pathsIndices$ and $\alpha\in\ProbabilityMeasureSet{\pathsIndices}$ such that:
\begin{align*}
    \policyProfile^\star &= \alpha \cdot (\path_k)_{k\in\pathsIndices}
    \\ \qquad\alpha_k &> 0 \qquad \forall k\in \pathsIndices^\star 
    \\ \alpha_k &= 0 \qquad \forall k\in \pathsIndices\backslash\pathsIndices^\star.
\end{align*}
Then $\cost(\policyProfile^\star, \policyProfile^\star) = \cost(\path_k, \policyProfile^\star)$ for all $k\in\pathsIndices^\star$.
In the mean field game, this translates in the equality of the travel time on all path $\path_k$ with $k\in\pathsIndices^\star$.
\end{proof}

Remark that in the $\NumOfPlayer$ player game, $\cost^\NumOfPlayer_\vehicle(\path_k, \policyProfile^\star_{-\vehicle}) = \cost^\NumOfPlayer_\vehicle(\path_\linkOne, \policyProfile^\star_{-\vehicle})$ does not translate into the equality of the trajectories following the pure policies $\path_k$ and $\path_\linkOne$ as the travel time depends on the vehicle $\vehicle$ location.
This is not the case in the mean field game.

\begin{proof}
\textsc{Proof of average deviation incentive of the Pigou mean field equilibrium policy in the corresponding N player game.}

For the point of view of one player, if the other players are following the mean field equilibrium policy, then the probability that $m$ other players are on the link $\linkTwo$ is ${N-1\choose m}0.5^{N-1}$.
In this case, the travel time on link $\linkTwo$ is $2 \frac{mT}{N}$ if the player does not use it, or $2 \frac{(m+1)T}{N}$ if the player uses it.
If the player follows the mean field equilibrium policy, with 50\% it will use the link $\linkOne$ and have a deviation incentive of $0.5 T \max \{0, 1-2 \frac{m+1}{N}\}$, and with 50\% it will use the link $\linkTwo$ and have a deviation incentive of $0.5 T \max \{0, 2 \frac{m+1}{N}-1\}$.
Therefore, the player deviation incentive is:
\begin{align*}
    \deviationIncentive
    &= 0.5 T \sum_{m=1}^{N-1} {N-1\choose m} 0.5^{N-1} \Big(.5 \max \{0, 2 \frac{m+1}{N}-1\} 
    \\
    &\qquad\qquad\qquad + 0.5 \max \{0, 1-2 \frac{m+1}{N}\}\Big)
    \\
    &= \frac{T}{N 2^{N}}\sum_{m=1}^N{N-1\choose m} \max\{\frac{N}{2}-m-1, m+1-\frac{N}{2}\}
\end{align*}

\end{proof}

\section{Average deviation incentive of the Braess mean field equilibrium policy in the corresponding $N$-player game}
In the Braess network, the mean field policy is almost a Nash equilibrium in the $\NumOfPlayer$-player game as soon as $\NumOfPlayer$ is larger than 30 players \cref{fig:braess_mean_field_in_n_player}.
The computation of the average deviation incentive is done with an explicit numeric computation, using a adaptive time discretisation.
The cost of a player in the mean field game is $3.75$, as shown on \cref{fig:braess_dynamics}, therefore a $0.05$ average deviation incentive is a excellent approximation of the $N$-player Nash equilibrium.

\begin{figure}[h]
    \centering
    \includegraphics[width=0.6\linewidth]{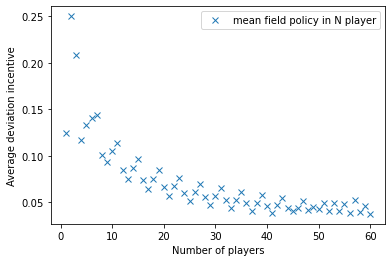}
    \caption{Average deviation incentive of the Nash equilibrium mean field policy in the $\NumOfPlayer$-player game as a function of $\NumOfPlayer$ in the case of the Braess game.
    The average deviation incentive is computed numerically explicitly using a adaptive time discretisation without OpenSpiel.
    }
    \label{fig:braess_mean_field_in_n_player}
\end{figure}
 
\section{Mean field game with player heterogeneity in their origin, destination and departure time}

The game evolution in the mean field equilibrium policy in the augmented Braess network game is shown on \cref{fig:aug_braess_dynamics}.

\begin{figure}[h]
    \centering
    \subfloat[]{\includegraphics[width=0.3\linewidth]{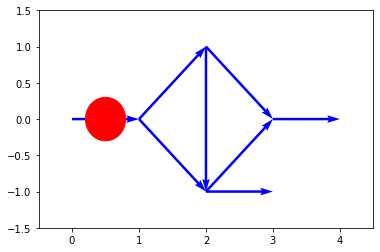}\label{subfig:ab_t_0}}
    \subfloat[]{\includegraphics[width=0.3\linewidth]{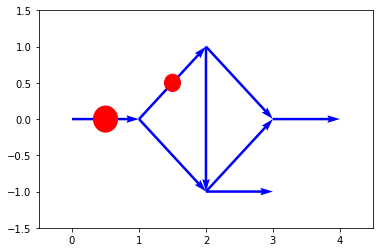}\label{subfig:ab_t_1}}
    \subfloat[]{\includegraphics[width=0.3\linewidth]{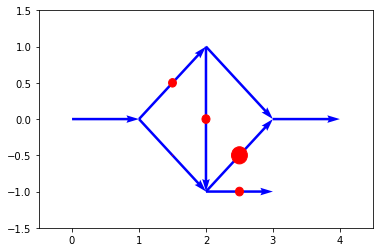}\label{subfig:ab_t_9}}
    
    \subfloat[]{\includegraphics[width=0.3\linewidth]{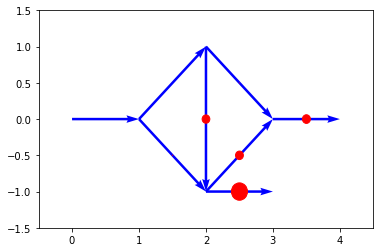}\label{subfig:ab_t_10}}
    \subfloat[]{\includegraphics[width=0.3\linewidth]{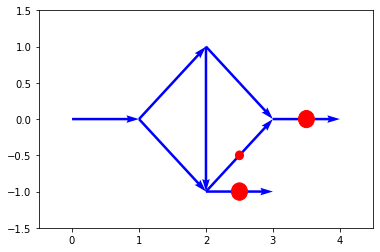}\label{subfig:ab_t_14}}
    \subfloat[]{\includegraphics[width=0.3\linewidth]{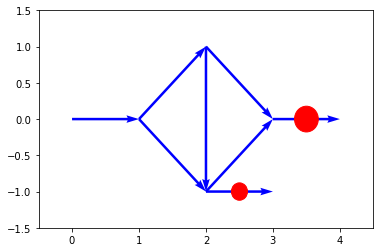}\label{subfig:ab_t_15}}
    \caption{Dynamics of the Braess network augmented with two destination in a set up with several departure time. Location of the cars at time 0.0 (\ref{subfig:ab_t_0}), 0.25 (\ref{subfig:ab_t_1}), 2.25 (\ref{subfig:ab_t_9}), 2.5 (\ref{subfig:ab_t_10}), 3.5 (\ref{subfig:ab_t_14}) and 3.75 (\ref{subfig:ab_t_15}). Vehicles that departs at the same time from the same origin and the same destination reach their destination at the same time, encoding the Nash equilibrium of the MFG.}
    \label{fig:aug_braess_dynamics}
\end{figure}
 
\section{Reproducing the experiments}

Experiment can be reproduced using the files in
the github \url{https://github.com/deepmind/open_spiel/tree/master/open_spiel/data/paper_data/routing_game_experiments} following the instructions in the readme.md file.

\end{document}